\documentclass[11pt,reqno]{amsart}
\usepackage{amsbsy,amsfonts,amsmath,amssymb,amscd,amsthm}
\usepackage[abs]{overpic}
\usepackage[displaymath,mathlines]{lineno}
\usepackage{mathrsfs,float,color}
\usepackage[mathcal]{euscript}
\usepackage{subfigure,enumitem,graphicx,epstopdf}
\usepackage[a4paper,text={6.1in,8in},centering,top=1.5in]{geometry}
\usepackage{pxfonts,epstopdf}
\usepackage[colorlinks=true,linkcolor=blue,citecolor=green]{hyperref}
\usepackage{srcltx}
\usepackage[margin=20pt,font=small,labelfont=bf,textfont=it]{caption}

\hypersetup{linktocpage}

\small\normalsize
\theoremstyle{plain}
\newtheorem{mainthm}{Theorem}

\newtheorem{thm}{Theorem}[section]

\newtheorem{lem}[thm]{Lemma}

\newtheorem{prop}[thm]{Proposition}

\newtheorem{defi}[thm]{Definition}

\theoremstyle{definition}
\newtheorem{exap}[thm]{Example}
\newtheorem{rem}[thm]{Remark}
\newtheorem{question}{Question}

\newcommand{\eqdef}{\stackrel{\scriptscriptstyle\rm def}{=}}

\DeclareMathOperator{\spn}{span} 
\DeclareMathOperator{\ind}{ind}

\begin{document}

\title{Robust heteroclinic tangencies}
\author[Barrientos]{Pablo G.~Barrientos}
\address{\centerline{Instituto de Matem\'atica e Estat\'istica, UFF}
   \centerline{Rua Prof. Marcos Waldemar de Freitas Reis, s/n, Niter\'oi, Brazil}}
\email{pgbarrientos@id.uff.br}
\author[P\'erez]{Sebasti\'an A.~P\'erez}
\address{\centerline{Centro de Matematica da Universidade do Porto}
   \centerline{Rua do Campo
Alegre 687, 4169-007 Porto, Portugal}} \email{sebastian.opazo@fc.up.pt}

\begin{abstract}
We construct diffeomorphisms in dimension $d\geq 2$ exhibiting
$C^1$-robust heteroclinic tangencies.
\end{abstract}

~\vspace{-0.75cm}
\\

\maketitle 
\thispagestyle{empty}

\section{Introduction}
An important problem in the modern theory of Dynamical Systems is
to describe  diffeomorphisms whose qualitative behavior exhibits
robustness under (small) perturbations and how abundant these sets
of dynamics can be. Motivated by this issue, Smale introduced
in~\cite{Sm67}  the hyperbolic diffeomorphisms as examples of
structural stable dynamics (open sets of dynamics which are all of
them conjugated). However, the transverse intersection between the
invariant manifolds of  basic sets was soon observed as a
necessary condition~\cite{Will70,Palis78,M88}.
The main goal of this article is to study the persistence of the
non-transverse intersection between those  manifolds. Namely, we
focus in tangencial heteroclinic orbits.

A diffeomorphism $f$ of a manifold $\mathcal M$  has a
\emph{heteroclinic tangency}  if there are  different tran\-sitive
hyperbolic sets $\Lambda$ and $\Gamma$, points $P\in \Lambda$,
$Q\in \Gamma$ and $Y\in W^u(P)\cap W^s(Q)$ such that
\[
     c_T \eqdef \dim \mathcal{M} - \dim [T_Y W^u(P) + T_Y W^s(Q)] > 0\quad \text{and} \quad
     d_T\eqdef\dim T_Y W^u(P)\cap T_Y W^s(Q)>0.
\]
The number $c_T$ is called \emph{codimension of the tangency} and
measures how far the tangencial intersection is from a transverse
intersection. On the other hand, $d_T$ indicates the number of
linearly independent common tangencial directions. Observe that
\[
   c_T=d_T- k_T   \qquad \text{with} \quad k_T=\ind(\Lambda)-\ind(\Gamma)
\]
where $\ind(\Sigma)$ denotes the stable index of a (transitive)
hyperbolic set $\Sigma$. The integer $k_T$ is called \emph{signed
co-index}. Notice that when $k_T>0$ this number coincides with the
classical co-index between $\Lambda$ and $\Gamma$. Moreover,
$k_T>0$ if and only if
\[
  \dim T_Y W^u(P) + \dim T_Y W^s(Q) > \dim M.
\]
If $k_T=0$, the heteroclinic tangency is called
\emph{equidimensional} and otherwise \emph{heterodimensional}.
Figure~\ref{fig:001} illustrates the different types of
heteroclinic tangencies in dimension three.

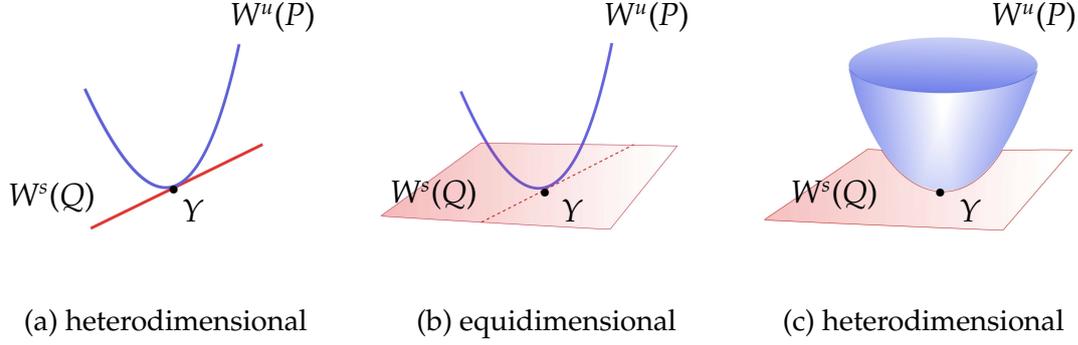
\begin{figure}
\centering
\begin{overpic}[scale=0.27,
]{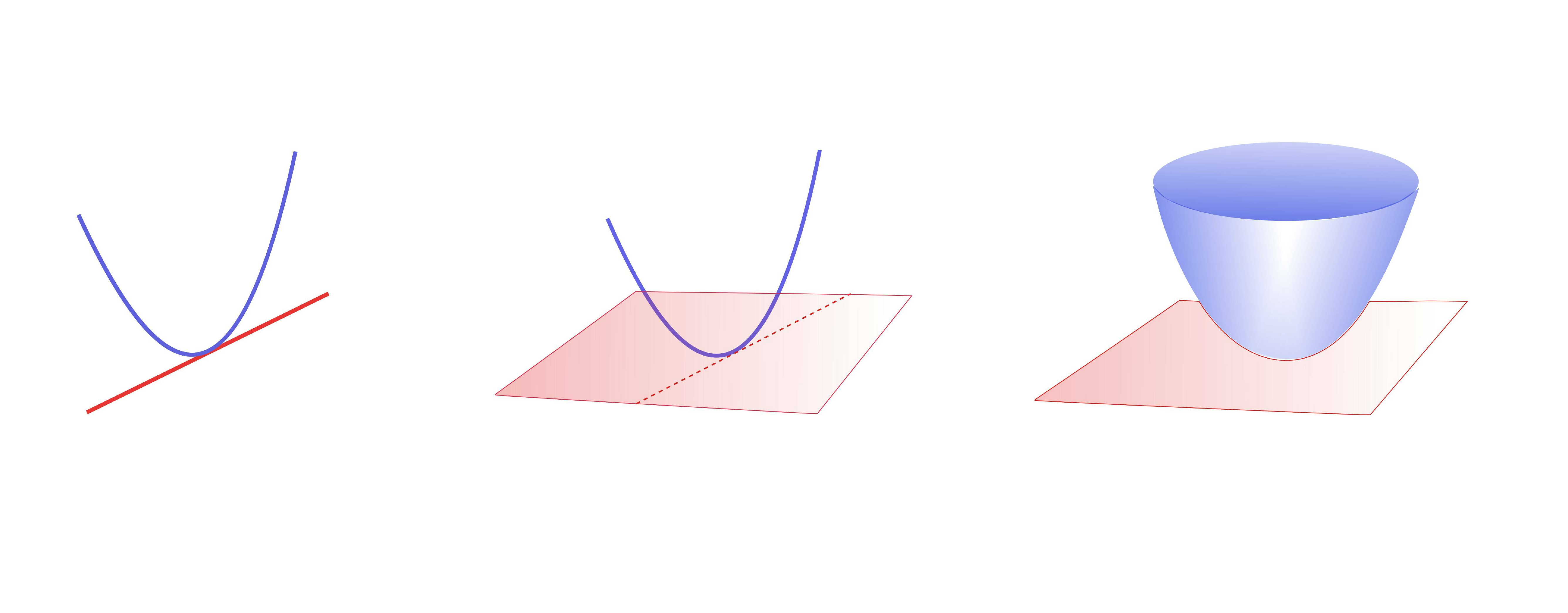} \scriptsize
         \put(-8,62){\large $W^s(Q)$}
            \put(54,68){\large \circle*{3}}
              \put(57,56){\large $Y$}
                   \put(75,130){\large $W^u(P)$}
\put(-2,15){\normalsize(a) heterodimensional} \put(145,15){\normalsize(b)
equidimensional} \put(282,15){\normalsize(c) heterodimensional}
          \put(135,63){\large $W^s(Q)$}
            \put(193,67){\large \circle*{3}}
              \put(199,56){\large $Y$}
                   \put(215,130){\large $W^u(P)$}
    \put(285,63){\large $W^s(Q)$}
            \put(341,67){\large \circle*{3}}
              \put(348,56){\large $Y$}
                   \put(360,130){\large $W^u(P)$}
\label{geometry}
 \end{overpic}
\caption{Heteroclinic tangencies in dimension 3: (a) $c_T=2$,
$d_T=1$ and $k_T=-1$; (b)  $c_T=1$, $d_T=1$ and $k_T=0$; (c)
$c_T=1$, $d_T=2$ and $k_T=1$.} \label{fig:001}
\end{figure}



Heterodimensional tangencies with signed co-index $k_T>0$ was
introduced in~\cite{DNP06} where interesting dynamics consequences
were obtained. Indeed, the authors showed that the $C^1$-unfolding
of a three dimensional heterodimensional tangency (with $k_T=1$)
leads to $C^1$-robustly non-dominated dynamics and
 in some cases to very intermingled dynamics related to universal dynamics,
 for details
see~\cite{DNP06,BonDia03}. In the $C^r$-topologies with $r>1$, the
bifurcation of such tangencies leads, for instance, to the
existence of blender dynamics~\cite{DiaKirShi:14,DiaPer:18}.
Kiriki and Soma in~\cite{KS12} obtain the first examples of
$C^2$-robust heterodimensional tangencies with $c_T=1$ and
$k_T=d-2$ in any manifold of dimension $d\geq 3$. Recently
in~\cite{BR17}  new examples of $C^2$-robust heterodimensional
tangencies with $0<c_T\leq \lfloor (d-3)/2 \rfloor$ and $1\leq k_T
\leq d-2-2c_T$ were also constructed in any manifold of dimension
$d\geq 5$. In the same work, $C^2$-robust heteroclinic tangencies
with $k\leq 0$ were also obtained. In~\cite{KS12} it was proposed
the problem of constructing $C^1$-robust heterodimensional
tangencies with $k>0$ in any dimension greater than $2$.
Motivated by this issue, the main result of this work shows that,
in particular, these tangencies can be built persistently under $C^1$-perturbations.

\begin{mainthm}\label{t.MT}
Every manifold of dimension $d\geq 2$ admits a diffeomorphism $f$
having a $C^1$-robust heteroclinic tangency of codimension $c_T=1$
and signed co-index $0\leq k \leq d-2$.
\end{mainthm}
By constraints of the dimension, in surfaces, we only get
equidimensional tangencies. In higher dimensions, we construct
both type of heteroclinic tangencies: equidimensional and
heterodimensional with all possible signed co-index between 0 and
$d-1$.

Theorem~\ref{t.MT} will be proved in Section~\ref{sec2}, by
providing a local construction close to the classical examples
given by Abraham and Smale~\cite{AS70}, Simon~\cite{Si72} and
Asaoka in \cite{Ao08}. In Section~\ref{sec3} we will give a
different proof of Theorem~\ref{t.MT} using ideas of the recent
work~\cite{BR19} studying the differential cocyle in the tangent
space. These new ideas allow us to generalize the construction for
large codimension ($c_T\geq 2$) in some particular cases. Namely,
we get the following result.
\begin{mainthm} \label{thmB}
Given integers $c_T \geq 1$ and $s> c_T$, there are
diffeomorphisms $f$ of the $d$-dimensional torus $\mathbb{T}^d$
with $d=c_T \cdot (s+1)$ having a $C^1$-robust heterodimensional
tangency of codimension $c_T$ and signed co-index $k_T=s-c_T>0$.
\end{mainthm}


The proof of the above theorem will be carried on in
Section~\ref{sec4}. Finally, in Section~\ref{sec5}  we conclude the work with a section of
open questions and future directions.
\section{Geometric construction of $C^1$-robust heteroclinic tangencies}
\label{sec2} Let $\Lambda$ be a Plykin attractor in a disc with
three holes~\cite{Rob98}. Let $Q$ be a saddle in the complement of
this disc as in Figure~\ref{fig:01}. To do possible the
construction we need to assume that $Q$ belongs to a Plykin
repellor $\Gamma$. This figure illustrates the two-dimensional
version of the Asaoka's argument~\cite{Ao08} (see
also~\cite{Si72}) providing a $C^1$-robust equidimensional
tangency in any surface between the stable manifold of $\Lambda$
and the unstable manifold of $Q$.
\begin{figure}
\centering
\begin{overpic}[scale=0.22,
]{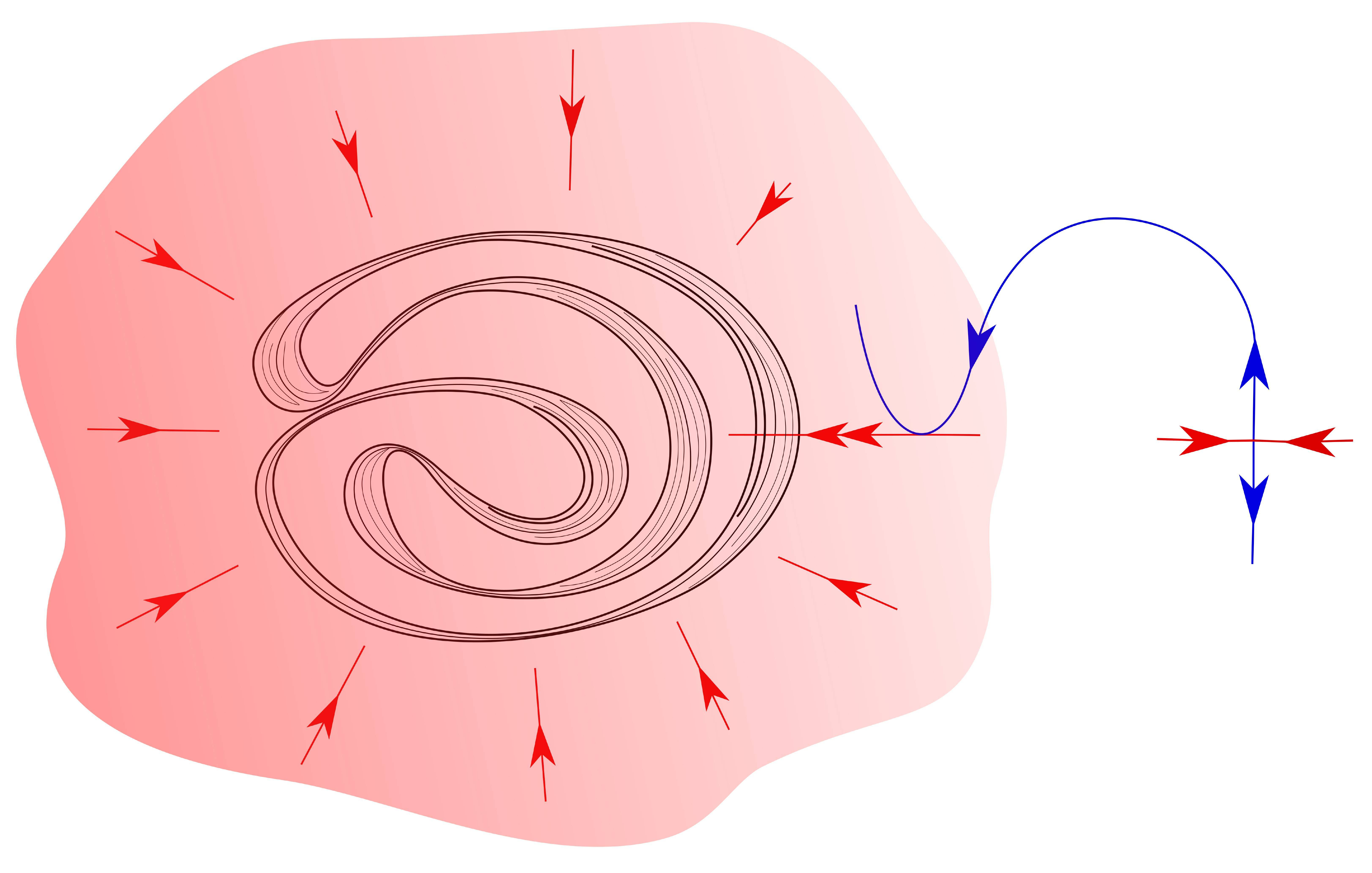} \scriptsize
            \put(65,158){\huge $\Lambda$}
    \put(201,110){\circle*{4}}
  \put(200,88){\large $P$}
   \put(232,110){\large \circle*{3}}
              \put(228,90){\large $y$}
\put(315,108){\circle*{4}}
  \put(324,88){\large $Q$}
\label{geometry}
 \end{overpic}
\caption{ Equidimensional tangency in dimension $d=2$}
\label{fig:01}
\end{figure}
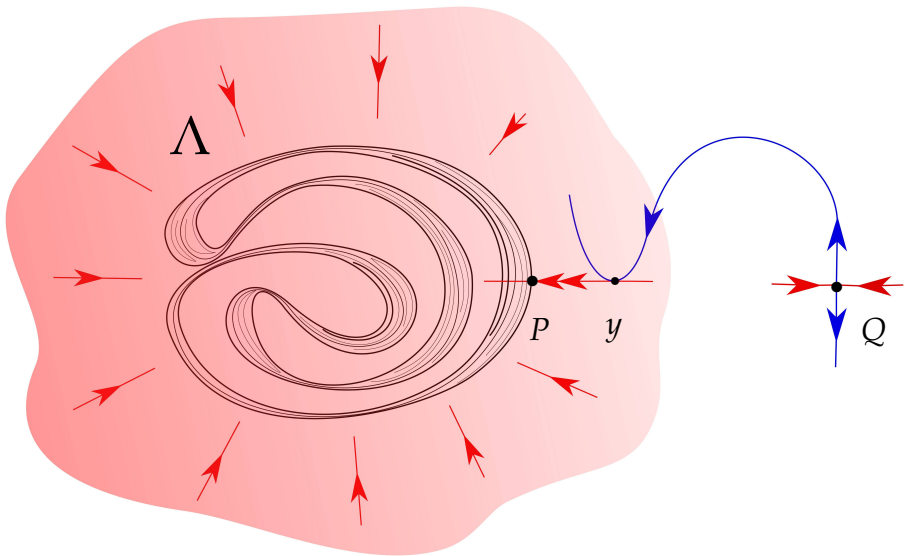

Using this idea, we built a diffeomorphism $f$  on any manifold
$\mathcal{M}$ of dimension $d\geq 2$ having a hyperbolic attractor
$\Lambda$  whose attracting region is a connected set and foliated
by $(d-1)$-dimensional stable submanifolds. After that, we
consider a   fixed point $P$  in $\Lambda$ and another
  fixed point $Q$ of $f$ of stable index $d-1-k$ where $0\leq
  k\leq d-2$ creating a heteroclinic tangency between $W^s(P)$ and $W^u(Q)$, so that
$W^u(Q)$ and $W^s(\Lambda)$ meet transversely. The  $C^1$-
persistence of this last intersection provides a $C^1$-robust
heteroclinic tangency associated with $\Lambda$ and $Q$.

\subsection{Construction} We now give the details of our construction. Since our
argument is local, we can put $\mathcal M=\mathbb{R}^d$ with $d
\geq 2$. First, we take a two-dimensional diffeomorphism $h$ with
a Plykin attractor $\Sigma$ constructed in local coordinates
inside  a disk
$D\subset \mathbb{R}^2$ with three holes. 
We consider a $C^r$-diffeomorphism $f:\mathcal M \to \mathcal M$
with $r\geq 1$ such that for a small $\varepsilon>0$, the
restriction of $f$ to the set $D_\varepsilon\eqdef
[-\varepsilon,\varepsilon]^{d-2}\times D$ is given by
\begin{equation}
\label{e.f}
  f=g\times h \quad \text{where} \quad g(t)=\lambda t \ \
  \text{for} \ \
  t\in [-\varepsilon,\varepsilon]^{d-2} \ \ \text{and} \ \
  0<\lambda<1.
\end{equation}
Thus, the set
\begin{equation}
\label{e.a} \Lambda\eqdef \{0^{d-2}\}\times
\Sigma=\bigcap_{n\geqslant 1}f^n(D_{\varepsilon})
\end{equation}
 is a hyperbolic
attractor of $f$ and  $D_\varepsilon$  is a trapping region of
$f$, i.e. $f(D_\varepsilon)\subset
\mathrm{interior}(D_\varepsilon)$. The structural stability of
$\Lambda$ provides the existence of a $C^1$-neighborhood
$\mathcal{V}$ of $f$ such that for each $g\in \mathcal{V}$, the
continuation $\Lambda_g$ of $\Lambda$ has by trapping region the
set $D_\varepsilon$.
We remark 
that the local stable manifolds $W^s_{loc}(x) =W^s(x)\cap
D_\varepsilon$ for $x\in\Lambda$  provide a foliation of the set
$D_\varepsilon$ by leaves (plaques) of dimension $d-1$. It is not
hard to verify that
this property also holds for any  diffeomorphism $g$
in~$\mathcal{V}$. We will denote by $W^s_{loc}(x,g)$ the stable
local manifold at $x$ for $g$.

Now we build the robust heteroclinc tangency of elliptic type.
Recall that a heteroclinic tangency $y\in W^u(Q)\cap W^s(P)$, is
of \textit{elliptic} type if there is a neighborhood $U$ of $y$
contained in either, $W^u(Q)$ or $W^s(P)$, say $W^u(Q)$, such that
any point in $U-\{y\}$ belongs to the same side of the tangent
space $T_y W^u(Q)$. We consider a fixed point $P\in \Lambda$ and a
small open  ball $B$ centered at $P$ such that $\overline{B}$ is
contained in $D_{\varepsilon}$. We observe that for every $g\in
\mathcal{V}$, $B$ is foliated by
\[\mathcal{F}_g(x)\eqdef W^s_{loc}(x,g)\cap B,\quad x\in\Lambda_g.\]
Consider a hyperbolic fixed point $Q \not \in D_\varepsilon$ of
$f$ with stable index $d-1-k$. By means of a homotopic
deformation, we force to the $(k+1)$-dimensional unstable manifold
$W^{u}(Q)$ intersects non-transversely the stable manifold
$W^{s}(P)$ in a heteroclinic tangency of elliptic type,
namely~$y$. Taking a suitable iterated if necessary, we can assume
that $y$ is in $B$.
\begin{figure}
\centering
\begin{overpic}[scale=0.7,
]{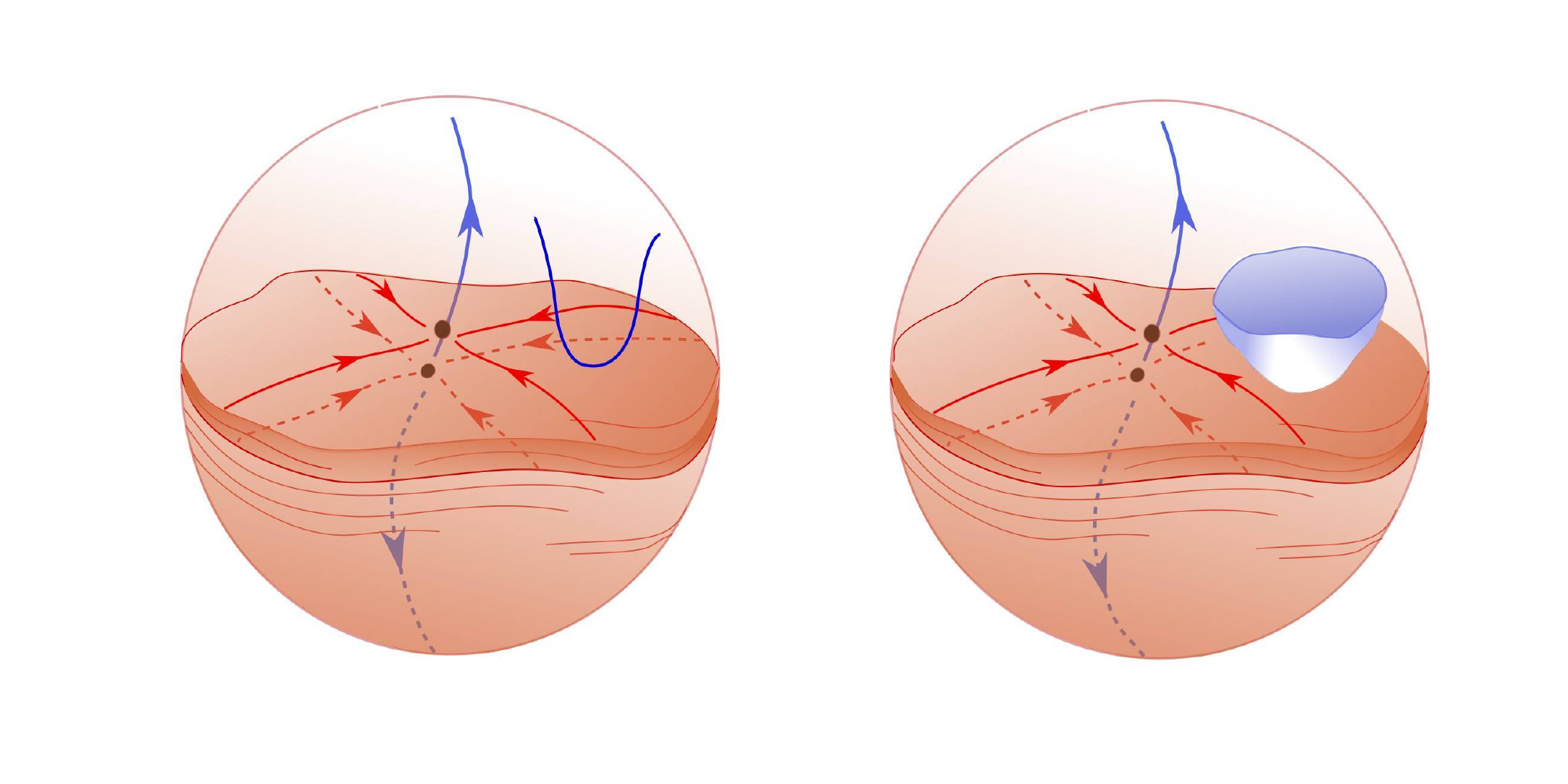} \scriptsize
        \put(108,100){\large $P_g$}
            \put(60,122){\large $\mathcal{F}_g(\bar{t})$}
               \put(163,150){\large $U_g$}
                  \put(162,108){\normalsize $y_g$}
                       \put(133,175){\large $\gamma_g(t)$}
    \put(165,117){\circle*{2}}
        \put(120,10){\large (a)}
       \put(330,175){\large $\gamma_g(t)$}
      \put(267,125){\large $\mathcal{F}_g(\bar{t})$}
    \put(360,135){\large $U_g$}
    \put(306,100){\large $P_g$}
    \put(360,101){\normalsize $y_g$}
    \put(362,108){\circle*{2}}
         \put(315,10){\large (b)}
\label{geometry}
 \end{overpic}
\caption{(a) $C^1$-robust equidimensional tangency. (b)
$C^1$-robust heterodimensional tangency.} \label{fig1}
\end{figure}
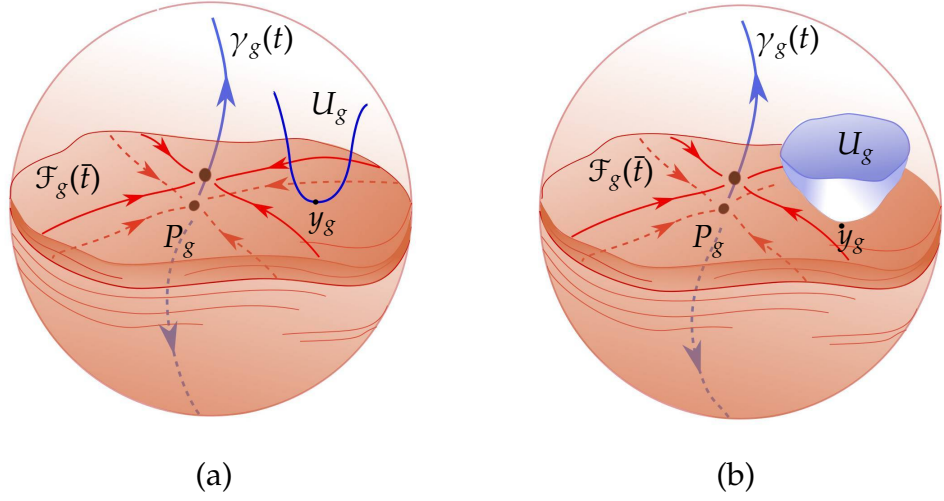
Thus, this last diffeomorphism, that again we call $f$, has a
heteroclinic tangency of codimension $c_T=1$ and signed co-index
$k_T=k$ with $0\leq k\leq d-2$, associated with the saddles $P\in
\Lambda$ and~$Q$.

On the other hand, by definition, there exists a neighborhood $U$
of $y$ contained in $W^{u}(Q)$ such that $U-\{y\}$ is contained in
$W^{s}(\Lambda)\pitchfork W^{u}(Q)$. See Figure~\ref{fig1}. We
will see that the $C^1$-persistence of these last transverse
intersections provides a $C^1$-robust heteroclinic tangency
associated with $\Lambda$ and $Q$.
Besides, for each $g\in \mathcal V$, we consider a small curve
$\gamma_g:t\in(-r,r)\mapsto \gamma_g(t)\in \Lambda_g\subset
D_{\varepsilon}$ parameterizing a small local unstable manifold of
the continuation $P_g=\gamma_g(0)$ of $P$ such that
\[B=\bigcup_{t\in(-r,r)}\mathcal F_g(t) \quad \text{where} \quad
\mathcal{F}_g(t)=\mathcal{F}_g(\gamma_g(t)).\] Since $y$ is a
heteroclinic tangency of elliptic type between $W^{u}(Q)$ and
$W^{s}(P)$ we can assume that (see Figure~\ref{fig1})
\[
\begin{split}
&U \cap \mathcal F_f(0)=\{y\},
\\
&U \cap \mathcal F_f(t)\subset W^{u}(Q)\pitchfork
W^{s}(\gamma_f(t)) \quad \mbox{for all $t\in (0,r)$},
\\
&U\cap \mathcal F_f(t)= \emptyset \quad \mbox{for all $t\in
(-r,0)$}.
\end{split}
\]
Our conditions imply that  for each $g\in \mathcal V$, the set
\[
I_g=\{t\in (-r,r): U_g\pitchfork \mathcal F_g(t) \neq \emptyset\}
\]
is inferiorly bounded where $U_g$ is a continuation in
$W^s_{loc}(Q_g)$ of the neighborhood $U$.
Thus, if $\bar{t}$ is the infimum of $I_g$ then $U_g$ and
$\mathcal F_g(\bar{t})$ meet in a heteroclinic tangency $y_g$ of
codimension $c_T=1$ and signed co-index $k_T=k$ with $0 \leq k\leq
d-2$. This completes the proof of Theorem~\ref{t.MT}.

\begin{figure}
\centering
\begin{overpic}[scale=0.7,
]{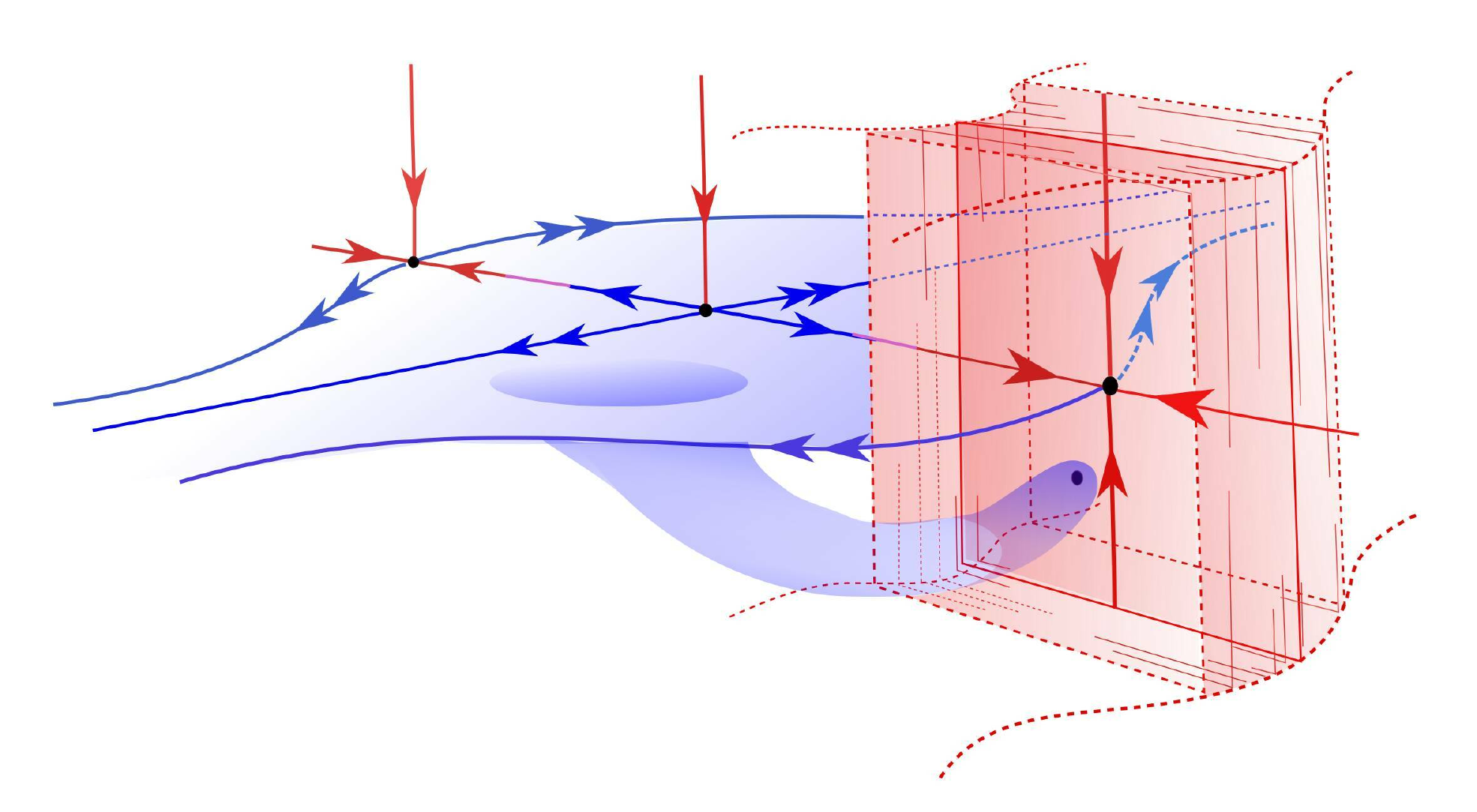}\scriptsize
        \put(282,125){\Large $P$}
            \put(172,146){\Large $Q$}
\label{geometry}
 \end{overpic}
\caption{Heterodimensional tangency constructed from a deformation
of diffeomorphism $f$ of $\mathbb{T}^3$ locally defined as the
product $g\times h$ where $h$ has a DA-attractor in $\mathbb{T}^2$
and $g$ is a contraction.} \label{fig32}
\end{figure}

\section{Differential construction of $C^1$-robust heteroclinic
tangencies} \label{sec3}

In this section we will prove again Theorem~\ref{t.MT} but now
using a different argument. This different approach allows us to
generalize the result to get robust heterodimensional tangencies
of large codimension in the next section. In order to explain the
idea behind of this new approach we will consider again the
situation described in Figure~\ref{fig:01}.

\begin{figure}
~\vspace{-0.75cm} \centering
\begin{overpic}[scale=0.74,
]{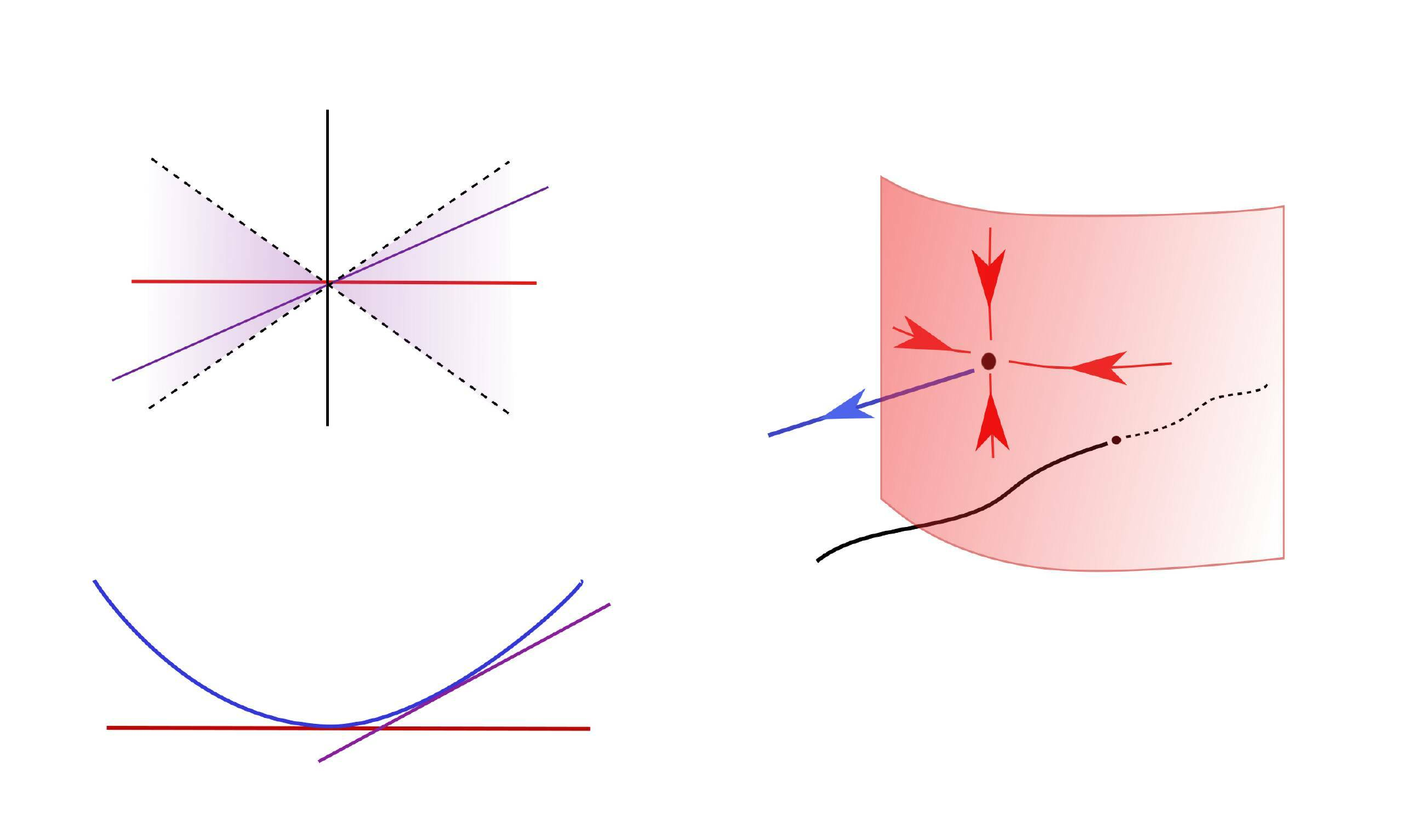} \scriptsize
                            \put(37,90){{\Large $\mathcal{S}\subset W^u(Q)$}}
                                    \put(200,30){{\Large $W^s(P)$}}
            \put(25,195){{\Large $\mathcal{C}^{s}_{\alpha}$}}
                     \put(185,215){\Large{$E$}}
                             \put(182,177){\Large{$E^s$}}
               \put(201,75){\Large{$T_x\mathcal{S}$}}
                      \put(105,36){\circle*{3}}
                            \put(103,45){\Large $y$}
                         \put(142,47){\circle*{3}}
                               \put(138,55){\Large $x$}
         \put(330,163){\normalsize{$(P,E^s)$}}
      \put(363,185){\Large{$W_{loc}^s(\Lambda^s)$}}
    \put(280,70){\Large $\mathcal{S}^s$}
   \put(362,115){\normalsize{$y^s=(y,E^s)$}}
\label{geometry}
 \end{overpic}
\caption{Figure on the left shows the folding manifold
$\mathcal{S}$ on $\mathcal{M}$ and how cover the cone
$\mathcal{C}^s_\alpha$ on $\mathbb{R}^d$. The other  shows the
transverse intersection between $\mathcal{S}^s$ and the stable
manifold of $\Lambda^s$ on $\mathbb{R}^d\times G(s,d)$}
\label{fig2}
\end{figure}
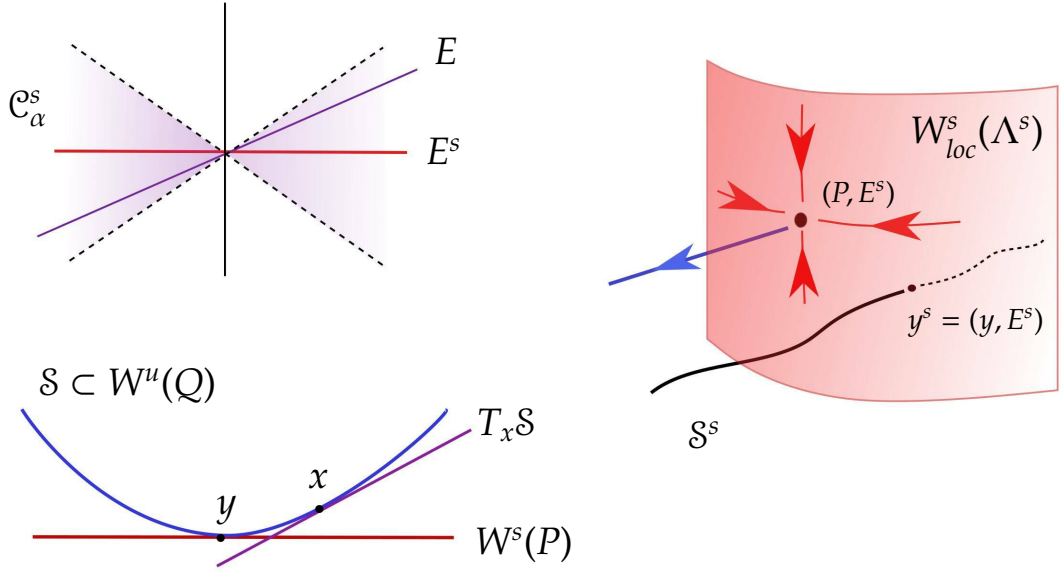

By considering, if necessary, local coordinates around $P$, define
the projective cocycle $f^s(x,E)=(f(x),Df(x)E)$ where $x \in
\mathbb{R}^2$  and $E$ belongs to the space $G(1,2)$ of
one-dimensional vector space in $\mathbb{R}^2$. Recall that the
$\Lambda$ is a hyperbolic attractor of $f$ with splitting
$E^s\oplus E^u$. Hence, the set $\Lambda^s=\Lambda \ltimes E^s
=\{(x,E): x\in \Lambda, \, E=E^s(x)\}$ is a hyperbolic set of
$f^s$ where the direction corresponding to the variable in
$G(1,2)$ is uniformly expanding. Thus,
$W^s(\Lambda^s)=W^s(\Lambda)\ltimes E^s=\{(x,E): x\in
W^s(\Lambda), E=E^s(x)\}$ is a two-dimensional manifold in the
three-dimensional space $\mathbb{R}^2\times G(1,2)$ as it is
showed in Figure~\ref{fig2}. On the other hand the unstable
manifold of $Q$ contains a folding manifold that we denote by
$\mathcal{S}$. That is, a small piece of the unstable manifold
contained the point $y$ in its interior. Namely, this manifold
folds with respect to the stable cone-field of $f$ at the point
$P\in\Lambda$ as it is represented Figure~\ref{fig2}. That is, by
considering linear transport to the origin of $\mathbb{R}^2$, the
union of tangent spaces $T_x S$ where $x\in \mathcal{S}$ cover the
cone $\mathcal{C}^s_\alpha=\{(x,y): |y|\leq \alpha |x| \}$ for
some small $\alpha>0$. This property allows us to see the set
$\mathcal{S}^s=S\ltimes TS=\{(x,E): x\in S, E=T_x\mathcal{S}\}$ as
a graph of a function $E\in \mathcal{C}^s_\alpha \mapsto x=x(E)\in
\mathcal{S}$. In other words, as a one-dimensional manifold in
$\mathbb{R}^2\times G(1,2)$ which is the image of a graph over
$G(1,2)$ and thus transversally intersecting $W^s(\Lambda^s)$ at
the point $(y,E^s)$ where $E^s=E^s(P)=\mathbb{R}\times\{0\}$.
Since this intersection is transversal, it persists for any small
perturbation. In particular, for any small perturbation $g$ of
$f$, we get a intersection point between $\mathcal{S}^s$ and
$W^s(\Lambda_g^s)$ where $\Lambda_g^s$ is the continuation of
$\Lambda^s$ for cocycle $g^s$ induced by $g$ in
$\mathbb{R}^2\times G(1,2)$.  Notice that this intersection point
between $\mathcal{S}^s$ and $W^s(\Lambda^s_g)$  provides the
tangency point and direction between $\mathcal{S}$ and a stable
manifold $W^s(z)$ for some $z\in \Lambda_g$. Therefore, we get a
robust tangency.

\subsection{Construction} Now we will give the formal details.
Recall the  $C^r$-diffeomorphism $f:\mathcal{M} \to \mathcal{M}$
in~\eqref{e.f} and the attractor $\Lambda=\{0^{d-2}\}\times
\Sigma$ in~\eqref{e.a}. This set has a well defined hyperbolic
structure $T_\Lambda\mathcal{M}=E^s\oplus E^u$ where the stable
bundle $E^s$ of $\Lambda$ is $(d-1)$-dimensional.
Observe that $E^s$ can be uniquely extended to a continuous
$Df$-invariant fiber bundle, which we also denote by $E^s$, over
each leaf $W^s_{loc}(x)$, $x\in \Lambda$, and so to the whole set
$D_\varepsilon$. Moreover, from the hyperbolicity of $\Lambda$, we
have that $E^{s}$ varies continuously with respect to the point
$x\in D_{\varepsilon}$ and  the diffeomorphism $g$ in a small
$C^1$ neighborhood $\mathcal{V}$ of~$f$.  Thus, for each $g\in
\mathcal{V}$ the set $D_\varepsilon$ is foliated by
$(d-1)$-dimensional (local) stable manifolds of $\Lambda_g$ which
are tangent to the bundle $E^s_g$ continuation of $E^s$.

Fix $s\eqdef d-1$. On the set $D_\varepsilon$ we have defined a stable cone-field
$\mathcal{C}^s_\alpha$ of dimension $s$ and size $\alpha>0$
 satisfying
\[
   E^s(x) \in \mathcal{C}^s_\alpha(x)\subset T_x\mathbb{R}^d \quad \text{and} \quad   Df^{-1}(x)\mathcal{C}^s_\alpha(f(x))\subset \mathcal{C}^s_\alpha(x)\quad \mbox{for all $x\in D_\varepsilon$}.
\]
In what follows, for notational simplicity, we omit  the subscript $\alpha$ in the notation $ \mathcal{C}^s_\alpha$.
The $s$-dimensional cone-field $\mathcal{C}^s$ can be seen as an
open set of the Grassmannian manifold
$G_{s}(\mathbb{R}^d)=\mathbb{R}^d\times G(s,d)$ where $G(s,d)$ is
set of the $s$-planes in $\mathbb{R}^d$.  Observe that in the case $d=2$, this Grassmannian manifold is the projective space.
Now consider the
differential cocycle induced by $f$ on $G_{s}(\mathbb{R}^d)$ given
by
\[
f^s: G_s(\mathbb{R}^d)\to G_s(\mathbb{R}^d), \quad
f^s(x,E)=(f(x),Df(x)E).
\]
Observe that $f^s$ is a $C^{r-1}$-diffeomorphism of
$G_s(\mathbb{R}^d)$ with $r\geq 2$.
Since $E^s$ is a repelling point of $Df$, 
\[
\Lambda^s=\Lambda \ltimes E^s \eqdef \{(x,E): x\in \Lambda, \
E=E^s(x) \}
\]
is hyperbolic set of $f^s$ with stable index equals to $\dim
E^s=s$. Namely, the splitting of $\Lambda^s$ is of the form
$E^s\oplus E^u\oplus E^{uu}$ where $E^s\oplus E^u$ corresponds
with the splitting of $\Lambda$ for $f$ and $E^{uu}$  with the
directions over $G(s,d)$. On the other hand, the local stable
manifold $W^s_{loc}(\Lambda^s)$ of $\Lambda^s$ contains the set
\[
D^s_\varepsilon=D_\varepsilon\ltimes E^s \eqdef \{(x,E):
x\in D_\varepsilon, \ E=E^s(x) \}.\] This is a manifold of
codimension the dimension of $G(s,d)$.

We now construct the heteroclinic tangency. First, we give a more
formal notion of heterodimensional tangency between any two
manifolds.

\begin{defi}
Let $\mathcal{L}$ and $\mathcal{N}$ be two submanifolds of
$\mathcal{M}$.
We say that $\mathcal{L}$ and $\mathcal{N}$ has a heteroclinic
tangency at $x\in \mathcal{L}\cap \mathcal{N}$ if $
  c_T= d_T - k_T >0
$
where
\[
   d_T=\dim T_x \mathcal{L} \cap T_x \mathcal{N} \quad \text{and}
   \quad  k_T=\dim \mathcal{L} + \dim  \mathcal{N} - \dim \mathcal{M}.
\]
The numbers $d_T=d_T(x,\mathcal{L},\mathcal{N})$,
$c_T=c_T(x,\mathcal{L},\mathcal{N})$ and
$k_T=k_T(x,\mathcal{L},\mathcal{N})$ are called, respectively,
dimension, codimension and signed co-index of the tangency between
$\mathcal{L}$ and $\mathcal{N}$ at $x$. The tangency is said to be
\emph{heterodimensional} if $k_T\not= 0$ and
\emph{equidimensional} if $k_T=0$.
\end{defi}

For simplicity and clarity of the exposition we restrict the
construction to the case of signed co-index $k_T=s-1=d-2$. By
means of a similar argument one can also get the other possible
co-index in Theorem~\ref{t.MT}. We will consider two types of
tangencies: elliptical (see Section~\ref{sec2}) and of saddle
type. We recall that a tangency $y\in W^u(Q)\cap W^s(P)$ is of
\textit{saddle} type if every neighborhood $U$ of $y$ contained in
either, $W^u(Q)$ or $W^s(P)$, say $W^u(Q)$, intersects each
connected component of $\mathbb{R}^d\setminus T_y W^u(Q)$.

\begin{exap}
Consider a diffeomorphism having two periodic saddles $P$ and $Q$
such that $0^d\in W^s_{loc}(P)\subset \mathbb{R}^{s}\times \{0\}$,
with $0^d\neq P$ and $\dim W^s(Q)=s$. Assume that,
 $\mathcal{S}([-1,1]^{s})\subset W^u(Q)$, where
\begin{equation} \label{eq1}
  \mathcal{S}:(t_1,\dots,t_s)\mapsto (t_1,\dots,t_s,
  t_1^2+\dots+t_ s^2)
  \end{equation}
  or
  \begin{equation} \label{eq1.1}
   \mathcal{S}:(t_1,\dots,t_s)\mapsto (t_1,\dots,t_s,
  t_1 t_2+\dots+t_{s-1} t_s).
\end{equation}
Then $0^d$ is a heteroclinic tangency of elliptic type between
$W^s(P)$ and $W^s(Q)$ choosing $\mathcal{S}$ as in~\eqref{eq1} and
the saddle type if $\mathcal{S}$ is as in~\eqref{eq1.1}.
\end{exap}

As it is usual, we identified the embedding $\mathcal{S}$ (as
those described above) with its image. Now, using  the
$s$-dimensional manifold $\mathcal{S}$ in~\eqref{eq1} and
~\eqref{eq1.1} we create  a tangency between the leaves the
foliation of $D_\varepsilon$ by stable manifold (of dimension $s$)
of $\Lambda$. Fix  a fixed point $P\in\Lambda$ and consider $y\in
W^s_{loc}(P)$ with $y\neq P$.
 Modifying slightly the construction of the attractor if necessary, we can
consider coordinates $(t_1,\dots, t_d)$ in neighborhood of $P$
such that
\begin{itemize}
\item
 $P$ is identified with $(1^{d-1},0)$ and $y$ with $0^d$,
 \item the local unstable manifold $W^u_{loc}(P)$ is $t_1=\dots=t_{d-1}=1$,
 \item for each $z=(1^{d-1},t)\in W^u_{loc}(P)$, the local stable manifold $W^s_{loc}(z)$ is $t_d=t$; and
\item the bundle $E^s$ is trivial on this neighborhood.
\end{itemize}
Hence, in this local coordinates we can assume that
$E^s=\mathbb{R}^s\times \{0\}$ and \[ \mathcal{C}^s=\{(u,v)\in
\mathbb{R}^s\oplus \mathbb{R}: \|v\|< \alpha\, \|u\|\}\cup
\{0^d\}\] where $\alpha>0$ is a small constant.

At this coordinates, the folding manifolds $\mathcal{S}$ in
~\eqref{eq1} and~\eqref{eq1.1} intersect $W^s(P)$ at $y$ in a
heteroclinic tangency of codimension $c_T=1$ and signed co-index
$k_T=s-1=d-2$. The next result state that this tangency persist
under perturbations.

\begin{prop} \label{prop1}
The folding manifold $\mathcal{S}$ has a heteroclinic tangency
with the stable foliation of $\Lambda$ which persists under small
$C^1$-perturbations of $f$.
\end{prop}

The proof of this proposition makes use of the following result:

\begin{lem} \label{lem1}
The set $\mathcal{S}^s= S\ltimes TS \eqdef \{(x,E): x\in
\mathcal{S}, \ E= T_x \mathcal{S} \}$ is a manifold of dimension
$\dim G(s,d)$ embedded as a disc in $\mathbb{R}^d\times G(s,d)$.
Namely it is a graph of a function of the form $E\in \mathcal{C}^s
\mapsto x=x(E)\in \mathcal{S}$.
\end{lem}

Let us postpone for a while the proof of lema, to conclude the proof of the proposition.

\begin{proof}[Proof of Proposition~\ref{prop1}]
Since $\mathcal{S}$ tangentially meets $W^s_{loc}(P)$ at $y$, we
get that $\mathcal{S}^s$ topologically transversally intersect
$W^s_{loc}(\Lambda^s)$  at $y^s=(y,E^s)$. The (topological)
transversality follows from Lemma~\ref{lem1} since $\mathcal{S}^s$
is a disc of $\dim G(s,d)$ and $W^{s}_{loc}(\Lambda^s)$ is a
manifold of codimension $\dim G(s,d)$. See Figure~\ref{fig2}.

Now consider a  diffeomorphism $g$ $C^1$-close to $f$. Observe
that the cocycle $g^s$ is a homeomorphism of ${G}_s(\mathbb{R}^d)$
only $C^0$-close to $f^s$. However, $\Lambda^s_g=\Lambda_g\ltimes
E^s_g$ is still a topological hyperbolic set for $g^s$ where
$\Lambda_g$ and $E^s_g$ are the continuation of $\Lambda$ and
$E^s$ for $g$. Thus, the set $W^s(\Lambda^s_g)$ contains a
manifold $C^0$-close to $W^s_{loc}(\Lambda^s)$ of codimension
$\dim G(s,d)$. Thus we still have a transversal intersection
between $\mathcal{S}^s$ and $W^s_{loc}(\Lambda^s_g)$. Observe that
if $(x,E)\in \mathcal{S}^s \cap W^s_{loc}(\Lambda^s_g)$ then $x\in
\mathcal{S}\cap W^s_{loc}(\Lambda_g)$, $E=E^s_g(x)$ and
$E=T_x\mathcal{S}$. Thus $E^s_g(x)=T_x\mathcal{S}$. This provides
a tangency between $\mathcal{S}$ and the stable foliation of
$\Lambda_g$ concluding the proof of the proposition.
\end{proof}

\begin{rem} Proposition~\ref{prop1} also holds for any small
enough $C^1$-perturbation of $\mathcal{S}$. To see this, if the
perturbation is $C^1$-close then we have a change of variable
$C^1$-close to the identity sending the perturbed manifold to the
folding manifold $\mathcal{S}$. Hence we get a new diffeomorphism
$g$ which is $C^1$-close to $f$. Thus, applying
Proposition~\ref{prop1} we get a tangency.
\end{rem}

Theorem~\ref{t.MT} follows from the above proposition and remark
by considering that the folding manifold is contained in the
unstable manifold of a hyperbolic fixed point of $f$ of unstable
index $s=d-1$. Observe that the codimension of the tangency  is
given by the formula $c_T=d_T-k_T$ where $d_T$ is the number of
tangent directions and $k_T$ is the co-indice between the
hyperbolic set involved. In this case, $c_T=s-(s-1)=1$ and
$k_T=s-1=d-2$.

To complete our construction we give the proof of  Lemma~\ref{lem1}.

\begin{proof}[Proof of Lemma~\ref{lem1}]
To prove that $\mathcal{S}^s$ is an embedded disc in
$\mathbb{R}^d\times G(s,d)$ we need to show that $\mathcal{S}^s$
is a graph of an injective function of the form
\[
   \mathcal{S}^s: E \in \mathcal{C}^s \mapsto x=x(E)\in G(s,d).
\]
To do this, we must associate to $E$ an unique point $x\in
\mathcal{S}$ such that $E=T_x\mathcal{S}$. In other words, we need
to show that
\[
   \mathcal{C}^s \subset \bigcup_{x\in \mathcal{S}}
   T_x\mathcal{S}.
\]
As above, we are standing that $\mathcal{C}^s$ is a small open set
in $G(s,d)$ centered at $E^s=\mathbb{R}^s\times \{0\}$ and the
tangent space $T_x\mathcal{S}$ as a vector space of
$\mathbb{R}^d$. Analytically, we need to solve the following
problem:  given $E \in \mathcal{C}^s$ we look for
$t=(t_1,\dots,t_s)$ such that $E=T_x\mathcal{S}$ where
$x=\mathcal{S}(t)$.

In order to do the calculation, we choose the elliptic form of the
folding manifold given in~\eqref{eq1}. For folding manifold of
saddle type in~\eqref{eq1.1} the argument is similar. Hence, 
\[
  T_x\mathcal{S}: (t'_1,\dots,t'_s) \in \mathbb{R}^s \mapsto
  (t'_1,\dots,t'_s,2t_1t'_1+\dots+2t_st'_s) \in \mathbb{R}^d,
  \quad \text{where $x=\mathcal{S}(t_1,\dots,t_s)$.}
\]
We write $E=\spn\langle v_1,\dots,v_s\rangle$ where
$v_i=(a_{1i},\dots,a_{di})$ for $i=1,\dots,s$. Hence
$E=T_x\mathcal{S}$ if, and only if, $v_i \in T_x\mathcal{S}$ for
all $i=1,\dots,s$. Equivalently, if
\[
    t'_{ji}=a^{}_{ji} \quad \text{for $j=1,\dots,s$ \ \ and} \ \ \
   2t^{}_{1}t'_{1i} + \dots + 2 t^{}_{s}t'_{si}= a^{}_{di} \quad
   \text{for all $i=1,\dots,s$}.
\]
Hence,
\begin{equation} \label{eq2}
2 \cdot
\begin{pmatrix}
a_{11} & \dots  & a_{1s} \\
       & \ddots       &    \\
a_{s1} & \dots  & a_{ss}
\end{pmatrix}
\cdot \begin{pmatrix}
t_{1} \\
\vdots \\
t_{s}
\end{pmatrix}
 =
\begin{pmatrix}
a_{d1} \\
\vdots \\
a_{ds}
\end{pmatrix}
\end{equation}
That is, we have a square linear system $A t = b$ where $A=A(E)$
and $b=b(E)$ depends on the vector space $E$. To find $t$ we need
to show that $A$ is an invertible matrix. To do this, we will take
as the vector space $E$ the center
$E^s=\mathbb{R}^s\times\{0\}=\spn\langle e_1,\dots e_s \rangle$ of
$\mathcal{C}^s$ where $e_i$  denotes the vector with a $1$ in the
$i$-th coordinate and 0's elsewhere.  We get in this case that
$A(E^s)=2 \cdot I_s$ where $I_s$ is the identity square matrix of
order $s$. Thus $\det A(E^s)\not = 0$. Then by the continuity for
all $E\in \mathcal{C}^s$ close to $E^s$ we uniquely
solve~\eqref{eq2} and thus we find $t=(t_1,\dots,t_s)$ such that
$E=T_x\mathcal{S}$ where $x=\mathcal{S}(t)$. This completes the
proof of the lemma.
\end{proof}

\section{$C^1$-robust heterodimensional tangencies of large codimension}
\label{sec4}

Fix $c_T\geq 1$ and $s> c_T$. Set $d=c_T\cdot (s+1)$. A hyperbolic
set $\Lambda$ of a diffeomorphism of a manifold $M$ is said to be
a \emph{codimension one expanding attractor} if for every $x\in
\Lambda$, holds that $W^u(x) \subset \Lambda$ and $\dim
W^u(x)=\dim M -1$. Let us take a codimension one expanding
hyperbolic attractor $\Lambda$ of a diffeomorphism $h$  on a
manifold of dimension $n=d-s+1$. In order to avoid the problem of
classifying the manifold that support these kind of attractors, we
set $\Sigma$ as the \emph{Derived from Anosov} (by short
$DA$-attractor) in the $n$-torus $\mathbb{T}^n$, see~\cite{Sm67}.
After that, we will consider a $C^r$ diffeomorphism $f$ of
$\mathbb{T}^d$ locally defined on
$D_\varepsilon=[-\varepsilon,\varepsilon]^{s-1}\times
\mathbb{T}^n$ for a fixed small $\varepsilon>0$ and $r\geq 2$ as
\[
  f=g\times h \quad \text{where} \quad g(t)=\lambda t \ \
  \text{for} \ \
  t\in [-\varepsilon,\varepsilon]^{s-1} \ \ \text{and} \ \
  0<\lambda<1.
\]
Notice that the set $\Lambda=\{0^{s-1}\}\times \Sigma$ is a
hyperbolic attractor of $f$ whose basin of attraction contains
$D_\varepsilon$. Moreover, $E^{s}=\mathbb{R}^{s-1}\times
\tilde{E}^{s}$ is the stable bundle of $\Lambda$ where
$\tilde{E}^s$ is the one-dimensional stable bundle of $\Sigma$ for
$h$. Thus, $s=\dim E^{s}$. Analogously as in  previous sections,
this bundle can be uniquely extended to a $Df$-invariant bundle
over $D_\varepsilon$ which we also denote by $E^s$. Consequently
the set $D_\varepsilon$ is foliated by $s$-dimensional stable
manifolds of $\Lambda$ which are tangent to
 $E^s$. This allows us
to consider a  stable cone-field $\mathcal{C}^s$
 of dimension $s$ defined in whole $D_\varepsilon$. As in
Section~\ref{sec3}, we defined the differential cocycle $f^s$
induced by $f$ on $G_s(\mathbb{R}^d)$. Similarly, we have that the
set $\Lambda^s_f=\Lambda \ltimes E^s$ is also a hyperbolic set of
$f^s$ with stable index equals to $s$ and whose local stable
manifold $W^s_{loc}(\Lambda^s_f)$ contains the set
$D^s_\varepsilon=D_\varepsilon\ltimes E^s$. Thus, this manifold
has by codimension the dimension of $G(s,d)$.


Restricting us to a small ball $B\subset D_\varepsilon$, we can
assume that the stable cone is give by
\begin{equation}\label{e.s-cones}
\mathcal{C}^s=\{(u,v)\in \mathbb{R}^s\oplus \mathbb{R}^{d-s}:
\|v\|<\alpha \|u\|\}
\end{equation}
where $\alpha>0$ is small enough and $E^s=\mathbb{R}^s\times
\{0^{d-s}\}$. We will consider a folding manifold $\mathcal{S}$ in
$B$ folded with respect to $\mathcal{B}$  which we introduce
formally as follows:

\begin{defi} A  manifold $\mathcal{S}$ of dimension $k\geq s$
 is called
\emph{folding manifold} in an open ball $B$ folded with respect to
the cone $\mathcal{C}^s$ if $\mathcal{S}\subset B$ and
 \[
\overline{\mathcal{C}^s} \subset \bigcup_{x\in \mathcal{S}}
   T_x\mathcal{S}.
\]
\end{defi}

We are understanding that $\overline{\mathcal{C}^s}$ is closure of
the open set $\mathcal{C}^s$ in $G(s,d)$ centered at $E^s$ which
we see as a cone in $\mathbb{R}^d$ and the tangent space
$T_x\mathcal{S}$ as a $k$-dimensional vector space of
$\mathbb{R}^d$. Taking $\alpha$ tends to zero we observe that the
above definition is in fact an infinitesimal property
of~$\mathcal{S}$. Thus without restriction we can assume that the
tangent space of $\mathcal{S}$ covers injectively the closure of
$\mathcal{C}^s$. This means that for every $E\in \mathcal{C}^s$ we
have a unique $x\in \mathcal{S}$ such that $E\leq T_x\mathcal{S}$.
Moreover, $x=x(E)$ varies continuously with respect to $E$.

\begin{exap}
Take $k=d-c_T=c_T\cdot s$.   Let us consider a $k$-dimensional
manifold $\mathcal{S}$ defined by
\begin{equation} \label{eq3}
  \mathcal{S}:(t_1,\dots,t_k)\mapsto (t_1,\dots,t_k,\,
  t_1^2,\, t_2^2, \dots, t^2_{c_T-1}, \,t^2_{c_T}+\dots+t_{k}^2).
\end{equation}
Hence, we have that
\[
  T_x\mathcal{S}: (t'_1,\dots,t'_s)  \mapsto
  (t'_1,\dots,t'_k,\,2t^{}_1t'_1,\dots,2t^{}_{c_T-1}t'_{c_T-1},\,2t^{}_{c_T}t'_{c_T}+\dots+2t^{}_kt'_k),
\]
where $x=\mathcal{S}(t_1,\dots,t_k)$. We write $E=\spn\langle
v_1,\dots,v_s\rangle \in \mathcal{C}^s$ where
$v_i=(a_{1i},\dots,a_{di})$ for $i=1,\dots,s$. Hence $E\leq
T_x\mathcal{S}$ if, and only if, $v_i \in T_x\mathcal{S}$ for all
$i=1,\dots,s$. Equivalently, if
\[
    t'_{ji}=a^{}_{ji} \quad \text{for $j=1,\dots,k$ \ \ and} \ \ \
  2t^{}_{\ell}t'_{\ell i} = a^{}_{k+\ell \, i}
  \quad \text{for $\ell=1,\dots,c_T-1$  \ \ and}
\]
\[
  2t^{}_{c_T}t'^{}_{c_T}+ \dots + 2 t^{}_{k}t'_{ki}= a^{}_{di}
 \quad
   \text{for all $i=1,\dots,s$}.
\]
This defines a linear system of $c_T \cdot s$ equations and $k$
variable. Since $k=c_T \cdot s$ we can write the system in the
form $At=b$ where $A=A(E)$ is a square matrix of ordem $k$ and
$b=b(E)$ is a vector in $\mathbb{R}^k$ depending  on the vector
space $E$. To find $t=(t_1,\dots,t_k)$ we need to show that $A$ is
an invertible matrix. To do this, we will take as the vector space
$E$ the center $E^s=\mathbb{R}^s\times\{0\}=\spn\langle e_1,\dots
e_s \rangle$ of $\mathcal{C}^s$ where $e_i$  denotes the vector
with a $1$ in the $i$-th coordinate and 0's elsewhere.  We get in
this case that $\det A(E^s)=2$.
Then, by  continuity, for all $E\in \mathcal{C}^s$ close to $E^s$
we uniquely solve the equation $At=b$ and thus we find
$t=(t_1,\dots,t_s)$ such that $E\leq T_x\mathcal{S}$ where
$x=\mathcal{S}(t)$. Therefore $\mathcal{S}$ is folding manifold
with respect to $\mathcal{C}^s$.
\end{exap}

As a consequence of the definition of folding manifold we get the
following lemma:

\begin{lem} \label{lem2}
Let $\mathcal{S}$ be a folding manifold folded with respect to
$\mathcal{C}^s$. Then the set
\[\mathcal{S}^s \eqdef \{(x,E):
x\in \mathcal{S}, \ E \leq T_x \mathcal{S} \ \text{with} \  \dim
E=s \}\] contains a manifold of dimension $\dim G(s,d)$ embedded
as a disc in $\mathbb{R}^d\times G(s,d)$.
\end{lem}
\begin{proof}
From the definition of folding manifold, we have an injective
continuous function $E \in \mathcal{C}^s \mapsto x\in \mathcal{S}$
such that $E\leq T_x\mathcal{S}$. This defines a subset of
$\mathcal{S}^s$ which is an embedding given by  $E \in
\mathcal{C}^s \mapsto (x,E) \in
   \mathcal{S}\times G(s,d)$ proving the lemma.
\end{proof}

The following result is the analogous to Proposition~\ref{prop1}.

\begin{prop} \label{pro2}
Let $\mathcal{S}$ be a folding manifold in $B$ of dimension
$k=d-c_T=c_T \cdot s$  folded with respect to $\mathcal{C}^s$.
Then $\mathcal{S}$ has a heterodimensional tangency of codimension
$c_T$ and signed co-index $k_T=s-c_T>0$ with the stable foliation
of $\Lambda$ which persists under small $C^1$-perturbations of
$f$.
\end{prop}

\begin{proof}
By assumption if $x\in B$ then
$E^s(x)=\mathbb{R}^s\times\{0^{d-s}\} \in \mathcal{C}^s$. Thus, we
have that $\mathcal{S}$ has a heterodimensional tangency of
codimension $c_T$ with $W^s(z)$ for some $z\in \Lambda$. Indeed,
by definition of the folding manifold $\mathcal{S}$ and the stable
bundle $E^s$ we find $x\in \mathcal{S}\subset B$ and $z\in
\Lambda$ such that $E^s(x) \leq T_x\mathcal{S}$ and
$E^s(x)=T_xW^s(z)$. Furthermore, the signed co-index of the
tangency is $k_T=s+k-d=k-1=s-c_T>0$ and the codimension is
$d_T-k_T=s-(s-c_T)=c_T$. On the other hand, the point $(x,E^s(x))$
belongs to $\mathcal{S}^c \cap W^s(\Lambda^s)$. Moreover, from
Lemma~\ref{lem2}, we have that $\mathcal{S}^s$  contains a disc of
dimension $\dim G(s,d)$. Additionally, $W^s(\Lambda^s)$ has
codimension $\dim G(s,d)$. Hence $\mathcal{S}^s$ transversally
intersect (in a topological sense) $W^s(\Lambda^s)$.

Arguing as in Proposition~\ref{prop1}, we still have a transversal
intersection between $\mathcal{S}^s$ and $W^s(\Lambda^s_g)$ for
any $C^1$-close diffeomorphism $g$ to $f$. Thus there is $(x,E)\in
\mathcal{S}^s \cap W^s(\Lambda^s_g)$. Then $x\in \mathcal{S}\cap
W^s(\Lambda_g)$, $E\leq T_x\mathcal{S}$ and $E=T_xW^s(z)$ for some
$z\in \Lambda_g$. Similar as above, this implies that
$\mathcal{S}$ and $W^s(z)$ has a heterodimensional tangency of
codimension $c_T$ and signed co-index $k_T=s-c_T$ concluding the
proof of the proposition.
\end{proof}

\begin{proof}[Proof of Theorem~\ref{thmB}]
It suffices to consider that the folding manifold in
Proposition~\ref{pro2} is contained in the unstable manifold of a
hyperbolic fixed point of $f$ of unstable index $k$.
\end{proof}
\section{Discussion and open questions}
\label{sec5} The goal of this paper was to construct heteroclinic
tangencies which are robust under $C^1$ perturbations. This
question was proposed in~\cite[pag.~3281]{KS12} where the authors
showed the existence of $C^2$-robust heterodimensional tangencies.
To approach this problem we have constructed $C^1$-robust
tangencies where one of the hyperbolic sets involved is an
attractor. This limitation prevents that our construction could be
carried on a heterodimensional cycle. A diffeomorphism has a
\emph{heterodimensional cycle} associated with two transitive
hyperbolic sets if these sets have different  indices (dimension
of the stable bundle) and their invariant manifolds meet
cyclically. This cycle is called \textit{non-transverse
(heterodimensional) cycle} if  besides its cyclic intersections
involves some heterodimensional tangency. In order to construct a
robust non-transverse heterodimensional cycle one must construct
the tangency involving hyperbolic sets which are not attractors.
This leads to our first question:

\begin{question} Is it possible
to construct $C^1$-robust non-transverse heterodimensional cycles?
\end{question}

Bearing in mind the classic constructions of robust homoclinic
tangencies  and heterodimensional cycles (\cite{New:79,BonDia:08})
via the unfolding of tangencies and cycles associated with
saddles, we ask the following:

\begin{question} Can a diffeomorphism $f$ having a
non-transverse heterodimensional cycle associated with saddles $P$
and $Q$ be $C^r$-approximated by a diffeomorphism $g$ with a
 $C^r$-robust non-transverse heterodimensional cycle associated
 with hyperbolic sets containing the continuations
 $P_g$ and $Q_g$ of $P$ and $Q$?
\end{question}

On the other hand, we also deal in this paper with the
construction of heterodimensional tangencies with signed co-index
$k>0$ of \emph{large} codimension. Robust tangencies of large
codimension were discovered in~\cite{BR17}. Namely, the authors
provided a method to construct $C^2$-robust \emph{bundle
tangencies} which are non-trivial intersection between different
fiber bundles. Bundle tangencies include homoclinic,
heterodimensional and equidimensional tangencies.
Recently in~\cite{BR19}, using similar ideas similar to this
paper, we have constructed new examples of robust homoclinic
tangencies of large codimension. The construction also uses an
abstract notion of folding manifold with respect to a cone-field
extending previous approach on robust homoclinic tangencies
in~\cite{BD12}. However, as in the case of this work, the
construction are limited to consider high dimensional manifolds.
The lower possible dimension that allows to have a homoclinic
tangency of large codimension is $d=4$. Similarly, $d=5$ is the
lower dimension to construct a large heterodimensional tangency
with signed co-index $k>0$. Thus we address the following
questions:

\begin{question}
Is it possible to build a robust heterodimensional tangency with
signed co-index $k>0$ (resp.~homoclinic tangency) of codimension
$c_T = 2$ in dimension $d=5$ (resp.~$d=4$)?
\end{question}

\subsection*{Acknowledgements}
We are grateful to Artem Raibekas for discussions and helpful
suggestions. During the preparation of this article PB was supported  by MTM2017-87697-P from Ministerio de
Econom\'ia y Competividad de Espa\~na and CNPQ-Brasil.
SP were partially supported by CMUP (UID/MAT/00144/2019), which is funded by FCT with national (MCTES) and European structural funds through the programs FEDER, under the partnership
agreement PT2020. SP also acknowledges financial support from a postdoctoral grant of the project PTDC/MAT-CAL/3884/2014


\end{document}